\documentclass{article}
\usepackage{inputenc,amsmath}
\usepackage{bbm,mathtools,amsthm,mathrsfs}
\usepackage[english]{babel}
\usepackage{graphicx}

\graphicspath{ {./images/} }
\newtheorem{theorem}{Theorem}[section]

\newtheorem{lemma}[theorem]{Lemma}
\newcommand{\defeq}{\vcentcolon=}
\newcommand{\pd}{\text{ .}}
\newcommand{\icol}[1]{% inline column vector
  \left(\begin{smallmatrix}#1\end{smallmatrix}\right)%
}

\DeclarePairedDelimiter\abs{\lvert}{\rvert}%
\makeatletter
\let\oldabs\abs
\def\abs{\@ifstar{\oldabs}{\oldabs*}}
\theoremstyle{definition}
\newtheorem{definition}{Definition}[section]

\title{Random walks, word metric and orbits distribution on the plane }
\author{Uriya Pumerantz}
\begin{document}

\maketitle
\tableofcontents
\section{Introduction}
Given a countably infinite group $G$ acting on some space $X$, an increasing family of finite subsets $G_n$ and $x\in X$, a natural question to ask is what asymptotical distribution the sets $G_nx$ form. More formally, we define for a function $f$ over $X$ the sums $S_n(f,x)=\sum_{g\in G_n}f(gx)$ and ask whether exists a function $\Psi(n):\mathbbm{N}\to\mathbbm{R}$ such that the sequence $\Psi(n)S_n(f,x)$ converges. This is a delicate problem that was studied under various settings \cite{Ledrappier,Nogueira,Gorodnik,Maucourant,MacWeiss}. The following work started with intentions of solving this problem for the linear action of lattices in $SL(2,\mathbbm{R})$ over $\mathbbm{R}^2$ when elements are chosen using a word metric. While not reaching a solution, some discoveries were made for the same problem in slightly different settings. These discoveries not only shed light in our initial problem, but are also quite interesting for their on sake, and are therefore brought here in detail.\\\\
We first study the action of a specific lattice in $PSL(2,\mathbbm{Z})$ on the projective line, with $G_n$ defined using a carefully chosen word metric. The asymptotic distribution is calculated and shown to be tightly connected to Minkowski's question mark function \cite{MinkFunc}, a fractal function which is usually studied in the field of Diophantine approximations. We proceed to show that the limit distribution is stationary with respect to a random walk on $G$ defined by a certain measure $\mu$. We further prove a stronger result stating that the asymptotic distribution is the limit point for any probability measure over the projective line pushed forward by the convolution power $\mu^{*n}$.\\\\
The work on the projective line shows that for certain random walks and word metrics the resulting asymptotical distribution is the same. But while a word metric raises algebraic difficulties, a random walk is sometimes simpler to handle. It is therefore reasonable to study random walks in order to draw conclusion regarding the word metric. The second part is devoted to the asymptotic distribution problem when elements are chosen using random walk driven by the action of a lattice in $SL(2,\mathbbm{R})$ acting on the plane. We show some calculations under very restrictive assumptions that offer partial solutions. While a decisive answer is not found, we offer a natural variant of the problem that seems both easier to solve and gives rise to an interesting object. We reach a solution for this variant that holds under specific conditions, and show numerical calculations which suggest that those conditions hold for the group studied in the first part of our work.\\\\
This research was conducted under the supervision of Prof. Barak Weiss, to whom I wish to extend my gratitude for a most resourceful guidance. The results in the second part have been accepted for publication in "Uniform Distribution Theory" journal. The results in the last part have not been submitted as they are still partial.
\section{Word metric, Farey group and the projective line}
\subsection{Settings and results}
The group $G=PSL(2,\mathbbm{R})$ has a natural action on the projective line $X=P(\mathbbm{R}^2)$ which stems from the linear action on $\mathbbm{R}^2$. For $g =[\icol{a \ \ b\\c\ \ d}]\in G$ and $[\icol{x\\y}]=x\in X$ the $G$ action on $X$ is defined by
$$gx=[\icol{ax+by\\cx+dy}]\pd$$
This action is well defined and does not depend of the choice of representatives in either $X$ or $G$. We shall explore a specific famous subgroup of $PSL(2,\mathbbm{Z})$.
\begin{definition}
\label{FGdef}
The subgroup of $PSL(2,\mathbbm{Z})$ generated by
$$a=[\icol{1\ \ -2\\1\ \ -1}] \ ;\ b=[\icol{0\ \ -1\\1\ \ 0}]\ \ ;\ c=[\icol{1\ \ -1\\2\ \ -1}] $$ is called the \textit{Farey group}.
\end{definition}
As briefly described in the introduction, we study asymptotic distribution of orbits, when elements are chosen using word metric. The following theorem states the main result for the current chapter.
\begin{theorem}
\label{WordLimit}
Let $\Gamma$ be the Farey group. Let $||\cdot||$ denote the word metric with respect to $\{a,b,c\}$ and set $\Gamma_n=\{\gamma\in\Gamma:||\gamma||=n\}$. For $x\in X = P(\mathbbm{R}^2)$, the projective line, $f:X\xrightarrow{}\mathbbm{R},n\in\mathbbm{N}$ we define
$$S_n(f,x)=\sum_{\gamma\in\Gamma_n}f(\gamma x)\pd$$
Then there exists a measure $\mu_{\bar{\mathcal{M}}}$ on $X$ such that for every $x\in X$ and every continuous $f$,
$$\lim_{n\to\infty}\frac{S_n(f,x)}{|\Gamma_n|}=\int_Xfd\mu_{\bar{\mathcal{M}}}\pd$$
\end{theorem}
The measure $d\mu_{\bar{\mathcal{M}}}$, named the extended Minkowski measure, is expressed explicitly in the next section. We proceed to show that the extended Minkowski measure is in fact stationary with respect to a specific random walk on $X$.
\begin{theorem}
\label{MinkStationary}
The extend Minkowski measure $\mu_{\bar{\mathcal{M}}}$ is stationary with respect to the random walk generated by $\mu(\{a\})=\mu(\{b\})=\mu(\{c\})=\frac{1}{3}$.
\end{theorem}
Stationary measures have great importance in dynamics. Particularly relevant to this work are results by Furstenberg \cite{Benoist} showing existence and uniqueness of stationary measure for certain random walks on projective spaces. While there is a general result guaranteeing the existence of such a measure, it is rare to be able to explicitly express one. In the particular setting studied here, using unique properties of the Farey group, the stationary measure is not only explicitly expressed, but also shown to have a connection to a function from a seemingly unrelated area. Notice that there are no $\Gamma$ invariant measures on $X$, thus the extended Minkowski measure is stationary but not invariant with respect to $\Gamma$.\\\\
Lastly, we show general conditions on a random walk under which the word metric limit and the stationary measure coincide. As a direct consequence of the proof we see that the extended Minkowski measure is also the limit measure of any probability measure on $X$ pushed forward by the nth convolution power
$$\mu^{*n}\defeq\underbrace{\mu*\mu*...*\mu}_{\text{n times}}\pd$$
\begin{theorem}
\label{MinkConv}
In the settings of Theorem \ref{MinkStationary}, for any probability measure $\pi$ on $X$, the following limit exists in the weak-$\ast$ topology:
$$\lim_{n\to\infty}\mu^{*n}*\pi=\mu_{\bar{\mathcal{M}}}\pd$$
\end{theorem}
Notice that for any random walk defined by a measure $\nu$ on a compact space with unique stationary measure, the Ces\`aro limit $\frac{1}{n}\sum_{k=1}^n\nu^{*k}*\pi$ converges in weak-$\ast$ topology to the stationary measure. However, the existence of a Ces\'aro limit does not imply Theorem \ref{MinkConv}. 
\subsection{Preliminaries}
\subsubsection{Farey group and tessellation}
We first describe a construction of the hyperbolic Farey tessellation $\mathscr{T}$. For $\frac{p_1}{q_1},\frac{p_2}{q_2}\in\mathbbm{Q}$  with $\gcd(p_i,q_i)=1$ we define $\frac{p_1}{q_1}\oplus\frac{p_2}{q_2}=\frac{p_1+p_2}{q_1+q_2}$. The ordered Farey sequences $\mathscr{F}_n=(s^n_1,s^n_2,...,s^n_{k(n)})$ are then constructed using the recurrence relation
$$\mathscr{F}_0=(0,1)=(s_1^0,s_2^0)\text{ ,}$$
$$\mathscr{F}_{n+1}=(s^n_1,s^n_1\oplus s^n_2,s^n_2,s^n_2\oplus s^n_3,s_3,...,s^n_{k(n)-1}\oplus s^n_{k(n)},s^n_{k(n)})\pd$$
The following lemma summarizes useful well known facts regarding the Farey sequences \cite{MinkFunc}:
\begin{lemma}
\label{FareySeq}
$ $
\begin{enumerate}
    \item $\cup_{n=1}^\infty \mathscr{F}_n=\mathbbm{Q}\cap[0,1]$.
    \item Every $q \in \mathbbm{Q}\cap [0,1]$ appears no more than once in any $\mathscr{F}_n$.
\end{enumerate}
\end{lemma}
Since for any Farey pair $p<q$, the inequality $p<p\oplus q<q$ holds, the ordered sequences $\mathscr{F}_n$ are in fact ordered using the usual order on the real line.
\theoremstyle{definition}
\begin{definition}
Two rationals $p,q\in\mathbbm{Q}$ with $p<q$ are called a \textit{Farey pair} if they are successive terms in some $\mathscr{F}_n$. That is, if exists $n\in\mathbbm{N}$ and $i\in\mathbbm{N}$ such that $p=s^n_i$ and $q=s^n_{i+1}$.
\end{definition} 
Let $\mathbbm{H}\defeq\{z\in\mathbbm{C}:Im(z)>0\}$ be the upper complex half plane equipped with the hyperbolic metric. For a detailed description of the hyperbolic upper half plane model see \cite{Beardon}. For any two points in the boundary $s,t\in\partial\mathbbm{H}=\mathbbm{R}\cup \{\infty\}$, $s\neq t$ we denote by $l(s,t)\subset\mathbbm{H}$ the unique infinite hyperbolic geodesic in $\mathbbm{C}$ that has $\{s,t\}$ in his closure and define $$\mathscr{T}_0=\bigcup_{\substack{(p,q) \\ \text{ Farey pair}} }l(p,q)\cup l(0,\infty)\cup (1,\infty)\text{ ,}$$
where the union is over all Farey pairs $(p,q)$. $\mathscr{T}_0$ is a set of boundary curves of a tessellation by ideal hyperbolic triangles of the region $\{z\in\mathbbm{H}:0\le Re(z) \le 1\}$. Notice that the word "triangle" is used here to describe both the edges of such element and the interior of it. The exact meaning is obvious from the context and should cause no confusion. To complete this to a tessellation of the entire hyperbolic plane we define $\tilde{\mathscr{T}}$ using integral translations of  $\mathscr{T}_0$,
$$\tilde{\mathscr{T}}=\bigcup_{n\in\mathbbm{Z}}(n+\mathscr{T}_0)\pd$$
One can check that $\tilde{\mathscr{T}}$ is a set of boundary curves of a tessellation of $\mathbbm{H}$. This tessellation is called the Farey tessellation. We denote $T(p,q,r)\defeq l(p,q)\cup l(q,r) \cup l(p,r)$ the ideal hyperbolic triangle with vertices at $-\infty<p<q<r\le \infty$. We refer to any ideal hyperbolic triangle $T(p,q,r)\subset \tilde{\mathscr{T}}$ as a "Farey tile" or simply as a "tile". We denote by $\mathscr{T}$ the set of all Farey tiles. Set $\Delta_e=T(0,1,\infty)\in\mathscr{T}$, and  let $\{a,b,c\}$ be as in definition \ref{FGdef}. One can check that $\{a,b,c\}$ are hyperbolic reflections on the edges of $\Delta_e$. By Poincar\'e's Theorem \cite{Maskit}, the representation of the Farey group $\Gamma$ in terms of generators and relations is $<a,b,c\ \ |\ \ a^2=b^2=c^2>$. The following lemma states that $\tilde{\mathscr{T}}$ can be described both in terms of Farey sequences and as $\Gamma$ orbit of $\Delta_e$.
\begin{lemma}
\label{lemma1}
 $$\tilde{\mathscr{T}}=\Gamma \Delta_e=\bigcup_{\gamma \in \Gamma}\gamma\Delta_e\pd$$
\end{lemma}
A proof of lemma \ref{lemma1} can be found in \cite{Series}.
\theoremstyle{definition}
\begin{definition}
Two tiles $\Delta_g,\Delta_h\in\mathscr{T}$ are called \textit{neighbors} if they share a common edge, that is $\Delta_g\cap\Delta_h=l(s,t)$ for some $s,t\in\partial\mathbbm{H}$.
\end{definition} 
\begin{lemma}
\label{AlgebraicNgb}
$\Delta_g$ is a neighbor of $\Delta_h$ if and only if $h=gs$ with $s\in\{a,b,c\}$.
\end{lemma}
\begin{proof}
If $h=gs$ then since $\Delta_s$ is a neighbor of $\Delta_e$ and since isometries move geodesics to geodesics, it follows that $\Delta_h=\Delta_{gs}=gs\Delta_e=g\Delta_s$ is a neighbor of $\Delta_g=g\Delta_e$. As each tile has exactly 3 neighbors and there are exactly 3 generators the condition is necessary. 
\end{proof} 
\subsubsection{Structure of the Farey tessellation}
In this section we prove some results regarding the structure of the group $\Gamma$.
\theoremstyle{definition}
\begin{definition}
Let $G$ be a group generated by a set $S\subset G$. Define the word metric on $G$ as follows: for any $g\neq e$ by $||g||_S=\min(\{n|g=s_1s_2...s_n,s_i\in S\})$ and $||e||_S=0$.
\end{definition}
Throughout this paper, we omit the subscript $S$ when the generating set is implied. The following general lemma is a direct consequence of corollary 1.4.8 in \cite{CoexterCombi}.
\begin{lemma}
\label{ReducedRep}
Let $G$ be a group defined in term of generators and relations by $<a_1,...,a_n\ \ |\ \ a_i^2=e>$. Then each $g\in G$ has a unique representation $g=s_1...s_n$ with $m=||g||$ and $s_i\in\{a_1,...,a_n\}$. This representation has $s_i\neq s_{i+1}$ for $1\le i < m$. 
\end{lemma}
One can see that $\Gamma$ satisfies the assumptions in Lemma \ref{ReducedRep}.
\begin{definition}
A representation that satisfies the conditions of Lemma \ref{ReducedRep} is called a \textit{reduced representation}.
\end{definition}
 The following lemmas summarize some important observations regarding the Farey tessellation.
\begin{figure}[h]
\centering
\includegraphics[width=\textwidth]{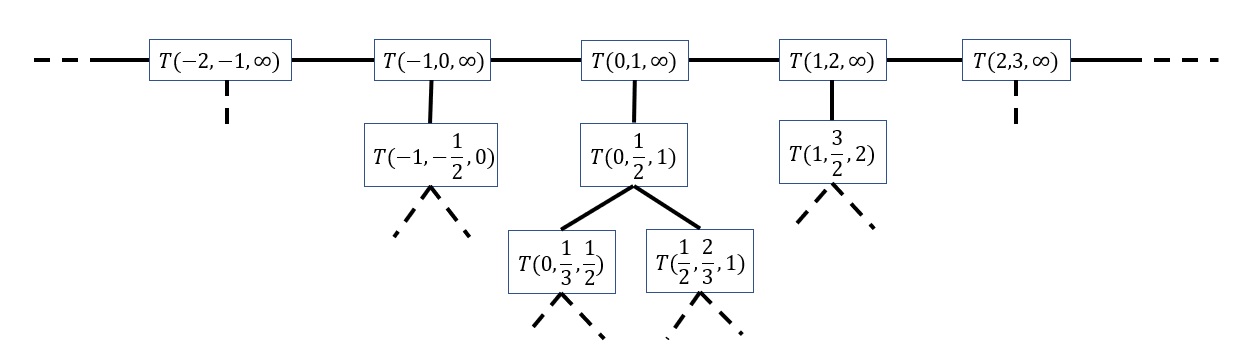}
\caption{An illustration of the Farey tessellation structure}
\end{figure}

\begin{lemma}
\label{NgbTiles}
Let $g\in\Gamma,g\neq e$ with a reduced representation $s_1...s_n$ and let $s\in\{a,b,c\}$ such that $s\neq s_n$. Then:
\begin{enumerate}
    \item If $\Delta_g=T(m,m+1,\infty)$ and $m>0$ then $\Delta_{gs_n}=T(m-1,m,\infty)$ and either $\Delta_{gs}=T(m+1,m+2,\infty)$ or $\Delta_{gs}=T(m,m\oplus (m+1),m+1)$.
    \item If $\Delta_g=T(m,m+1,\infty)$ and $m<0$ then $\Delta_{gs_n}=T(m+1,m+2,\infty)$ and either $\Delta_{gs}=T(m-1,m,\infty)$ or $\Delta_{gs}=T(m,m\oplus (m+1),m+1)$.
    \item If $\Delta_g=T(q_1,q_2,q_3)$ then either $\Delta_{gs_n}=T(q_1,q_3,r)$ or $\Delta_{gs_n}=T(r,q_1,q_3)$ or $\Delta_{gs_n}=T(q_1,q_3,\infty)$ with $r\in\mathbbm{Q}$ such that $q_3 = q_1\oplus r$ or $q_1=r\oplus q_3$ respectively, and either $\Delta_{gs}=T(q_1,q_1\oplus q_2,q_2)$ or $\Delta_{gs}=T(q_2,q_2\oplus q_3,q_3)$.
\end{enumerate}
Notice that for every $g\in\Gamma$ either $\Delta_g=(m,m+1,\infty)$ or $\Delta_g=T(q_1,q_2,q_3)$ with $m\in\mathbbm{Z}$ and $q_i\in\mathbbm{Q}$. The case $m=0$ has $\Delta_g=\Delta_e$, which makes it trivial.
\end{lemma}
\begin{proof}
We will show a proof for the third case only. The other two cases are proved using identical reasoning. It follows from Lemma \ref{AlgebraicNgb} that the neighbors of $\Delta_g$ are $\Delta_{ga},\Delta_{gb},\Delta_{gc}$. The construction of the Farey sequences, as described in terms of Farey sequences, implies that those tiles correspond to $T(q_1,q_1\oplus q_2,q_2),T(q_2,q_2\oplus q_3,q_3)$ and either $T(q_1,q_3,r)$ or $T(r,q_1,q_3)$ or $T(q_1,q_3,\infty)$ with $r\in\mathbbm{Q}$ such that $q_3 = q_1\oplus r$ or $q_1=r\oplus q_3$ respectively. Assume by contradiction $T_{gs_n}=T(q_1,q_1\oplus q_2,q_2)\defeq T(q_1^1,q_2^1,q_3^1)$. Then, since $s_n\neq s_{n-1}$ either $T_{gs_ns_{n-1}}=T(q_1,q_1\oplus(q_1\oplus q_2),q_1\oplus q_2)\defeq T(q_1^2,q_2^2,q_3^2)$ or $T_{gg_ng_{n-1}}=T(q_1\oplus q_2, (q_1\oplus q_2)\oplus q_2, q_2)\defeq T(q_1^2,q_2^2,q_3^2)$. Notice that $||gs_n||=||g||-1$. In a similar way we may keep shortening $g$ until reaching $\Delta_{gs_n...s_1}=\Delta_e=T(0,1,\infty)= T(q_1^n,q_2^n,q_3^n)$. The intervals $[q_1^i,q_3^i]$ form a descending filtration and thus $1-0=q_3^n-q_1^n<q_3-q_1<1$, arriving at a contradiction. By same method we see $T_{gs_n}\neq T(q_2,q_2\oplus q_3,q_3)$, so the only possibility is that $\Delta_{gs_n}$ is either $T(q_1,q_3,r),T(r,q_1,q_3)$ or $T(q_1,q_3,\infty)$. The rest of the claim regarding the two other generators follows immediately.
\begin{figure}[h]
\centering
\includegraphics[width=\textwidth]{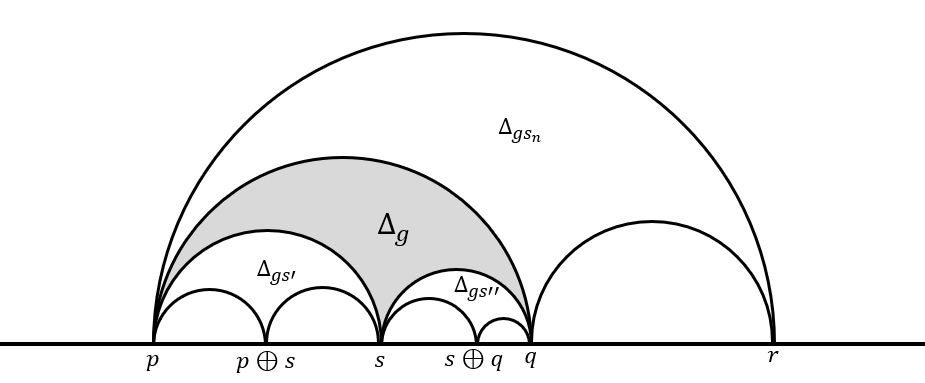}
\caption{A tile with finite vertices}
\end{figure}
\begin{figure}[h]
\centering
\includegraphics[width=\textwidth]{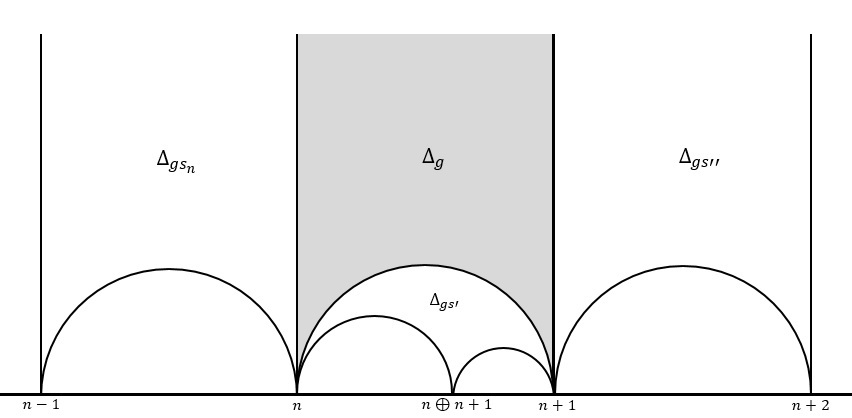}
\caption{A tile with vertex at $\infty$}
\end{figure}
\end{proof}
\begin{lemma}
\label{TilesContain}
Let $g,h \in \Gamma$ with reduced representations $g=g_n...g_1,h=h_m...h_1$ and $n\le m$. Denote $\Delta_g=T(p_1,p_2,p_3)$ and $\Delta_h=T(q_1,q_2,q_3)$ with $p_3,q_3<\infty$. Then $[q_1,q_3]\subset [p_1,p_3]$ if and only if $h_i=g_i$ for all $1\le i \le n$.
\end{lemma}
\begin{proof}
Assume $h_i=g_i$ for all $1\le i \le n$. Then
$$\Delta_h=h_m...h_{n+1}h_n...h_1\Delta_e=h_m...h_{n+1}g\Delta_e=h_m...h_{n+1}\Delta_g\pd$$
Since $h_i\neq h_{i+1}$ for all $i$ the claim follows from Lemma \ref{NgbTiles}.\\
Let $2\le i_0 \le n$ be the first integer such that  $h_{i_0}\neq g_{i_0}$. Then $h_{i_0-1}...h_1\Delta_e=g_{i_0-1}...g_1\Delta_e=T(r_1,r_2,r_3)$. For simplicity assume $r_i<\infty$, then since $h_{i_0}\neq g_{i_0}$ and both are different from $g_{i_0-1}=h_{i_0-1}$ we get by Lemma $\ref{NgbTiles}$ that $g_{i_0}...g_1\Delta_e=T(r_1,r_1\oplus r_2,r_2)$ and $h_{i_0}...h_1\Delta_e=T(r_2,r_2\oplus r_3,r_3)$. Since $h,g$ are both given as reduced representation lemma \ref{NgbTiles} finishes the proof. The other cases where $r_3=\infty$  are proven using lemma \ref{NgbTiles} in a similar manner. The case $i_0=1$ is trivial.
\end{proof}
The next lemma gives a characterization of some Farey tiles in terms of Farey sequences.
\begin{lemma}
\label{WordFareyLength}
Let $g\in\Gamma$ such that $||g||=n>0$ and $\Delta_g=T(p,r,q)$ with $0\le p < r < q \le 1$. Then $r$ first appears in $\mathscr{F}_n$ and $p,q$ are a Farey pair in $\mathscr{F}_{n-1}$. 
\end{lemma}
\begin{proof}
We prove by induction on $||g||$. Let $g=s_1...s_n$ be the reduced representation of $g$. By Lemma \ref{NgbTiles} for $s\neq s_n$, $\Delta_{gs}=T(p,p\oplus r,r)$ or $\Delta_{gs}=T(r,r\oplus q,q)$. Lemma \ref{FareySeq} implies that both $p \oplus r$ and $r \oplus q$ first appear in $\mathscr{F}_{n+1}$. For the base case $||g||=1$, since $0\le p,r,q \le 1$ it follows that $\Delta_g=T(0,\frac{1}{2},1)$. Since $\frac{1}{2}\in\mathscr{F}_1$ and $0,1\in\mathscr{F}_0$ we are done. 
\end{proof}
The next lemma extends Lemma \ref{WordFareyLength} to all tiles.
\begin{lemma}
\label{WordFareyLength2}
Let $g\in\Gamma$ with $||g||=n>0$ and $\Delta_g=T(p,r,q)$. Then either $p=n,r=n+1,q=\infty$ or $p=-n,r=1-n,q=\infty$ or exists $m\in\mathbbm{Z}$ such that $p-m,q-m$ are Farey pair in $\mathscr{F}_{n-m-1}$ and $r-m=(p-m)\oplus(q-m)$
\end{lemma}
The proof is left as excercise for the reader.
\begin{lemma}
\label{CombiCount}
Let $\Gamma_n\defeq \{\gamma\in\Gamma:||\gamma||=n\}$. For $n \in \mathbbm{N},n>0$:
$$|\Gamma_n|=3\cdot2^{n-1}$$
and $|\Gamma_0|=1$.
\end{lemma}
\begin{proof}
 For $n=0$, $\gamma=e$ is the only possible word hence $|\Gamma_0|=1$. For any $n\ge1$ we use Lemma \ref{ReducedRep} and count reduced representations. if $\gamma=s_1s_2....s_n$ is a reduced representation, it has $s_1\in\{a,b,c\}$ and for every $i \ge 2$, $s_i\in\{a,b,c\}\setminus\{s_{i-1}\}$. Therefore $|\Gamma_n|=3\cdot 2^{n-1}$.
\end{proof}
\subsubsection{Minkowski function and measure} \label{MinkSec}
The Minkowski question mark function was first constructed by Hermann Minkowski and is studied in the field of Diophantine approximations. It is traditionally labeled by "?" but for readability purposes we label it throughout this paper by $\mathcal{M}$. If $[a_0 ; a_1,a_2,...,a_n]$ is the continued fraction representation of $x\in \mathbbm{Q}$ then 
$$\mathcal{M}(x)=a_0+\sum_{k=1}^n\frac{(-1)^{k+1}}{2^{a_1+...+a_k}}\text{ .}$$
If $x=[a_0 ; a_1,a_2...]$ is irrational then the summation becomes infinite. We briefly describe an equivalent construction which will be more useful for our needs. $\mathcal{M}$ is first defined as a function from $\mathbbm{Q}\cap[0,1]$ to the dyadic rationals $\mathbbm{Q}_2\cap[0,1]$ and then extended to all of $[0,1]$ using continuity arguments. For more details see \cite{MinkFunc}. If $q$ is the $(k+1)$-th term in $\mathscr{F}_n$, that is $\mathscr{F}_n = (q_1,q_2,...,q_k,q,q_{k+2},...,q_m)$ then $\mathcal{M}(q)\defeq \frac{k}{2^n}$. Notice that every $q\in\mathbbm{Q}$ appears in infinitely many $\mathscr{F}_n$ but the definition does not depend on the choice of $n$. The next lemma, whose proof is immidate from the preceding discussion, will be useful in the next section.
\begin{lemma}
\label{FareyPair}
Let $p<q$ be a Farey pair in $\mathscr{F}_n$. Then $\mathcal{M}(q)-\mathcal{M}(p)=\frac{1}{2^n}$.
\end{lemma}
$\mathcal{M}$ is an ascending continuous bounded function defined on $\mathbbm{Q}$. It can be used to construct a measure on $[0,1]$ by defining $\mu_\mathcal{M}([a,b))=\mathcal{M}(b)-\mathcal{M}(a)$.
Our purposes demand extending $\mu_\mathcal{M}$ to the entire real line. The following extension will turn out to be useful. For $q\in[n,n+1]$ define
$$\bar{\mathcal{M}}(q)=\frac{1}{3}(\sum_{k=-\infty}^{n-1}\frac{1}{2^{|k|}}+\frac{\mathcal{M}(\{q\})}{2^{|n|}})\text{ ,}$$
where $\{q\}$ denotes the fractional part of $q$.   Using the same methods used for $\mathcal{M}$ we construct the extended Minkowski measure $\mu_{\bar{\mathcal{M}}}$. Notice that $\lim_{q\to-\infty}\bar{\mathcal{M}}(q)=0$ and $\lim_{q\to\infty}\bar{\mathcal{M}}(q)=1$ thus $\mu_{\bar{\mathcal{M}}}$ is a probability measure.\\\\
Notice that the tight connection between the Farey tessellation and continued fractions is not new. Series has shown in \cite{Series} that the continued fraction expansion of any $x\in\mathbbm{R}$ can be read from its position relative to the Farey tessellation. It follows from her work that if $T(p,q,s)$ is a tile with $p,q,s<\infty$ and $q=[a_0;a_1,...,a_n]$ then $p\oplus q,q\oplus s\in \{[a_0;a_1,...,a_n+1],[a_0;a_1,...,a_n-1,2]\}$. Therefore, for any tile $T(p,q,s)$ with finite vertices, if $p = [a_0;a_1,...,a_n]$ and $p\oplus q=[b_0;b_1,...,b_m]$ then $\sum_{i=0}^mb_i=1+\sum_{i=0}^na_i$. This suggests an alternative approach for proving some claims presented here.
\subsection{Main results}
\subsubsection{Word metric}
The projective line $X$ can be identified with $\partial\mathbbm{H}=\mathbbm{R}\cup \{\infty\}$ using $[\icol{x\\y}]\xrightarrow{}\frac{x}{y}$ when $y\neq 0$ and $[\icol{1\\0}]\xrightarrow{}\infty$. $G$ acts on $\partial \mathbbm{H}$ with Mobius transformations. By choosing suitable representatives we see that $\partial\mathbbm{H}$ and $X$ are in fact isomorphic $G$-sets:
$$g[\icol{x\\y}] = g[\icol{\frac{x}{y}\\1}]=g[\icol{z\\1}]=[\icol{az+b\\cz+d}]=[\icol{\frac{az+b}{cz+d}\\1}]\pd$$ $\partial \mathbbm{H}$ will be more convenient to work with for our needs. We first prove Theorem \ref{WordLimit} for $z\in\{0,1,\infty\}$.
\begin{lemma}
For every $f\in C(X)$ and for $z\in\{0,1,\infty\}\subset X$,
$$\lim_{n\to\infty}\frac{S_n(f,z)}{|\Gamma_n|}=\int_Xfd\mu_{\bar{\mathcal{M}}}\pd$$
\end{lemma}
\begin{proof}
Let $f=\mathbbm{1}_{[p,q]}$ with $p,q\in\mathbbm{Q}\cap[0,1]$ a Farey pair. The summation $S_n(f,z)$ can be expressed as $S_n(f,z)=|\{\gamma\in\Gamma_n:\gamma z\in[p,q]\}|$. Since $p,q$ are a Farey pair they are vertices of some triangle $T(p,s,q)\in\mathscr{T}$. Let $g\in\Gamma$ such that $T(p,s,q)=\Delta_g$ and let $N\defeq||g||$. Let $\gamma\in\Gamma$ with $\Delta_\gamma = T(u,v,w)$. Since elements of $\Gamma$ move vertices of tiles to vertices of tiles, and since $\{0,1,\infty\}$ are the vertices of $\Delta_e$, $\gamma z$ is a vertex of $\Delta_\gamma$. Therefore
$\gamma z\in [p,q]$ implies either $[u,w]\subset [p,q]$ or $w=p$ or $u=q$. The typical case is $[u,w] \subset [p,q]$ and for every $n$ exist at most 2 different $\gamma$ such that the other cases occur. Using Lemma \ref{TilesContain} and a simple combinatorial argument we deduce that for every $n>N$, $S_n(f,z)=2^{n-N}+\theta(n)$
with $\theta(n)\in\{0,1,2\}$. Using Lemma \ref{CombiCount} we get
$$\lim_{n\to\infty}\frac{S_n(f,z)}{|\Gamma_n|}=\lim_{n\to\infty}\frac{2^{n-N}+\theta(n)}{3\cdot 2^{n-1}}=\frac{1}{3}2^{1-N}\pd$$
Since $||g||=N$ and using Lemma \ref{WordFareyLength} we see that $p,q$ are Farey pair in $\mathscr{F}_{N-1}$. Lemma \ref{FareyPair} implies $\mathcal{M}(q)-\mathcal{M}(p)=2^{1-N}$ and therefore
$$\lim_{n\to\infty}\frac{S_n(f,z)}{|\Gamma_n|}=\frac{1}{3}(\mathcal{M}(q)-\mathcal{M}(p))=\bar{\mathcal{M}}(q)-\bar{\mathcal{M}}(p)\pd$$
If $p,q\in\mathbbm{Q}\cap[m,m+1],m\in\mathbbm{Z}$ we use similar arguments. This time to compute $S_n(f,z)=|\{\gamma\in\Gamma_n:\gamma z\in[p,q]\}|$ notice that $\gamma z\in[p,q]$ implies $\gamma z\in[m,m+1]$, therefore $\Delta_\gamma=T(u,v,w)$ has $[u,w]\subset[m,m+1]$ or $w=m$ or $u=m+1$. If the reduced representation of $\gamma$ is $\gamma_1...\gamma_t$ and $[u,w]\subset[m,m+1]$ it must have $\Delta_{\gamma_1...\gamma_{|m|}}=T(m,m+1,\infty)$. Using same counting method as before we see
$$\lim_{n\to\infty}\frac{S_n(f,z)}{|\Gamma_n|}=\frac{1}{3}\frac{1}{2^{|m|}}(\mathcal{M}(\{q\})-\mathcal{M}(\{p\}))=\bar{\mathcal{M}}(q)-\bar{\mathcal{M}}(p)\pd$$ 
Now let $f\in C(X)$. Since $X$ is compact we can find a sequence of simple functions $f_m$ which converge uniformly to $f$. That is, there exists a sequence of simple functions $f_m=\sum c_i\mathbbm{1}_i$ with $\mathbbm{1}_i=\mathbbm{1}_{[p_i,q_i]}$ indicator functions such that $\epsilon(m)\defeq \sup_{x\in X}|f_m(x)-f(x)|$ has $\lim_{m\to\infty}\epsilon(m)=0$. Lemma \ref{FareySeq} implies density of Farey pairs in $[0,1]$ so we can take $p_i,q_i$ to be integral translations of Farey pairs and get
$$\lim_{n\to\infty}\abs{\frac{S_n(f,z)}{|\Gamma_n|}-\int_Xfd\mu_{\bar{\mathcal{M}}}}=$$
$$\lim_{n\to\infty}\abs{\frac{S_n(f,z)}{|\Gamma_n|}-\frac{S_n(f_m,z)}{|\Gamma_n|}+\frac{S_n(f_m,z)}{|\Gamma_n|}-\int_Xf_md\mu_{\bar{\mathcal{M}}}+\int_Xf_md\mu_{\bar{\mathcal{M}}}-\int_Xfd\mu_{\bar{\mathcal{M}}}}\le$$
$$\lim_{n\to\infty}\Big(\abs{\frac{S_n(f-f_m,z)}{|\Gamma_n|}}+\abs{\frac{S_n(f_m,z)}{|\Gamma_n|}-\int_Xf_md\mu_{\bar{\mathcal{M}}}}+\abs{\int_X(f-f_m)d\mu_{\bar{\mathcal{M}}}}\Big)\le$$
$$\lim_{n\to\infty}\Big(\epsilon(m)+\abs{\frac{S_n(f_m,z)}{|\Gamma_n|}-\int_Xf_md\mu_{\bar{\mathcal{M}}}}+\epsilon(m)\Big)=2\epsilon(m)\pd$$
We may now take $m\to\infty$ and get
$$\lim_{n\to\infty}\abs{\frac{S_n(f,z)}{|\Gamma_n|}-\int_Xfd\mu_{\bar{\mathcal{M}}}}=\lim_{m\to\infty}\lim_{n\to\infty}\abs{\frac{S_n(f,z)}{|\Gamma_n|}-\int_Xfd\mu_{\bar{\mathcal{M}}}}\le$$
$$\lim_{m\to\infty}2\epsilon(m)=0\pd$$
\end{proof}
We are now ready to prove the main theorem.
\begin{proof}[Proof of theorem \ref{WordLimit}]
Let $f\in C(X)$ and $z\in X = \mathbbm{R}\cup\{\infty\}$. We show
$$\lim_{n\to\infty}\frac{1}{|\Gamma_n|}\sum_{\gamma\in\Gamma_n}|f(\gamma z)-f(\gamma 0)|=0\pd$$
We first divide $\Gamma_n$ into four disjoint sets. For $\delta,R,L>0$ define
$$\Gamma_n^\delta=\{[\icol{a \ \ b\\c\ \ d}]\in\Gamma_n:|z+\frac{d}{c}|<\delta\}\text{ ,}$$
$$\Gamma_n^{\delta,R,L+}=\{[\icol{a \ \ b\\c\ \ d}]\in\Gamma_n:R>|z+\frac{d}{c}|\ge\delta,|c|<L\}\text{ ,}$$
$$\Gamma_n^{\delta,R,L-}=\{[\icol{a \ \ b\\c\ \ d}]\in\Gamma_n:R>|z+\frac{d}{c}|\ge\delta,|c|\ge L\}\text{ ,}$$
$$\Gamma_n^R=[\{\icol{a \ \ b\\c\ \ d}]\in\Gamma_n:|z+\frac{d}{c}|\ge R\}\text{ ,}$$
and denote $\Gamma_n^{\delta,R,L\pm}\defeq\Gamma_n^{\delta,R,L+}\cup \Gamma_n^{\delta,R,L-}$. These sets are well defined as both $\frac{d}{c}$ and $|c|$ are the same for $\pm\icol{a \ \ b\\c\ \ d}$. Notice that $\gamma = [\icol{a \ \ b\\c\ \ d}]$ has $\gamma\frac{-d}{c}=\infty$ hence $\gamma^{-1}\infty = \frac{-d}{c}$.  Using Lemma \ref{WordLimit} with $z=\infty$ and the fact that $||\gamma||=||\gamma^{-1}||$ we get $$\lim_{n\to\infty}\frac{|\Gamma_n^\delta|}{|\Gamma_n|}=\lim_{n\to\infty}\frac{|\{\gamma\in\Gamma_n:|z-\gamma^{-1}\infty|<\delta\}|}{|\Gamma_n|}=$$
$$\lim_{n\to\infty}\frac{|\gamma\in\Gamma_n:\gamma^{-1}\infty\in[z-\delta,z+\delta]|}{|\Gamma_n|}=\lim_{n\to\infty}\frac{S_n(\mathbbm{1}_{[z-\delta,z+\delta]},\infty)}{|\Gamma_n|}=$$
$$\mu_{\bar{\mathcal{M}}}([z-\delta,z+\delta])$$ We apply same reasoning for $\Gamma_n^\delta$ and get $\lim_{n\to\infty}\frac{|\Gamma_n^R|}{|\Gamma_n|}=\mu_{\bar{\mathcal{M}}}([z+R,\infty)\cup[z-R,-\infty))$. $f$ is continuous on compact space and therefore bounded by some $B\in\mathbbm{R}$ therefore
$$\frac{1}{|\Gamma_n|}\sum_{\gamma\in\Gamma_n^\delta}|f(\gamma z)-f(\gamma 0)|+\frac{1}{|\Gamma_n|}\sum_{\gamma\in\Gamma_n^R}|f(\gamma z)-f(\gamma 0)|\le$$
$$B(\mu_{\bar{\mathcal{M}}}([z-\delta,z+\delta])+\mu_{\bar{\mathcal{M}}}([z+R,\infty)\cup[z-R,-\infty)))\xrightarrow[\delta\to0]{}0\pd$$
The convergence being due to continuity of $\bar{\mathcal{M}}$. To bound the sum over $\Gamma_n^{\delta,R,L\pm}$ we approximate $|\gamma z - \gamma 0|$. Let $\gamma = [\icol{a \ \ b\\c\ \ d}]\in\Gamma_n^{\delta,R,L\pm}$ and assume $d\neq 0$. If $d=0$ then $\gamma0=\infty$ and $\Delta_\gamma=T(m,m+1,\infty)$ with some $m\in\mathbbm{Z}$. There are at most 2 such $\gamma$ in $\Gamma_n$. Notice that $\gamma\in\Gamma_n^{\delta,R,L\pm}$ implies $cz+d\neq 0$ so we can write
$$|\gamma z - \gamma 0|=|\frac{az+b}{cz+d}-\frac{b}{d}|=|\frac{z}{dc(z+\frac{d}{c})}|\le \frac{|z|}{\delta |c|}\pd$$
 Every $\gamma\in\Gamma_n^{\delta,R,L-}$ has $\gamma^{-1}=-\frac{d}{c}\in[z-R,z+R]$. Lemma \ref{WordFareyLength2} implies that for each $n$ there are at most 2 different $\gamma\in\Gamma_n$ that can have $\gamma_1\infty = \gamma_2\infty$. Since $\gamma^{-1}\infty\in \mathbbm{Q}$ and $\abs{c}<L$ we can bound
$$|\Gamma_n^{\delta,R,L-}|\le 2|\gamma^{-1}\infty:\gamma\in\Gamma_n^{\delta,R,L-}|\le 2|\{\frac{p}{q}\in\mathbbm{Q}\cap[z-R,z+R]:q<L\}|\le$$
$$2(2R+1)(L+(L-1)+...+1)< 2(2R+1)L^2\pd$$
If $\gamma\in\Gamma_n^{\delta,R,L+}$ it has $c\ge L$ hence $|\gamma z-\gamma 0|\le \frac{|z|}{\delta L}$ and uniform continuity of $f$ then implies $|f(\gamma z) - f(\gamma 0)|<\epsilon(L)\xrightarrow[L\to \infty]{}0$.
Putting everything together we get:
$$\lim_{n\to\infty}|\frac{S_n(f,z)}{|\Gamma_n|}-\frac{S_n(f,0)}{|\Gamma_n|}|\le\lim_{n\to\infty}\frac{1}{|\Gamma_n|}\sum_{\gamma\in\Gamma_n}|f(\gamma z)-f(\gamma0)|=$$
$$\lim_{n\to\infty}\frac{1}{|\Gamma_n|}(\sum_{\gamma\in\Gamma_n^\delta}|f(\gamma z)-f(\gamma0)|+\sum_{\gamma\in\Gamma_n^{\delta,R,L\pm}}|f(\gamma z)-f(\gamma0)|+\sum_{\gamma\in\Gamma_n^R}|f(\gamma z)-f(\gamma0)|)\le$$
$$B\mu_{\bar{\mathcal{M}}}([z-\delta,z+\delta]+B\mu_{\bar{\mathcal{M}}}([z+R,\infty)\cup[z-R,-\infty))+\epsilon(L)\pd$$
We can now take $L,R\to\infty$ and $\delta\to 0$ and get
$$\lim_{n\to\infty}|\frac{S_n(f,z)}{|\Gamma_n|}-\frac{S_n(f,0)}{|\Gamma_n|}|=\lim_{\delta\to 0}\lim_{R\to \infty}\lim_{L\to \infty}\lim_{n\to\infty}|\frac{S_n(f,z)}{|\Gamma_n|}-\frac{S_n(f,0)}{|\Gamma_n|}|\le$$
$$\lim_{\delta\to 0}\lim_{R\to \infty}\lim_{L\to \infty}(B\mu_{\bar{\mathcal{M}}}([z-\delta,z+\delta]+B\mu_{\bar{\mathcal{M}}}([z+R,\infty)\cup[z-R,-\infty))+\epsilon(L))=0\pd$$
\end{proof}
\subsubsection{Stationary measure and random walk average}
We now prove Theorem \ref{MinkStationary}, stating that the extended Minkowsi probability measure is in fact stationary with respect to the random walk defined by $\mu(\{a\})=\mu(\{b\})=\mu(\{c\})=\frac{1}{3}$. Notice that by Furstenberg's uniqueness Theorem \cite{Benoist} the stationary measure in this case is unique.
\begin{proof}
Let $f$ be a continuous function over $X$. Then
$$\mu * \mu_{\bar{\mathcal{M}}}(f)=\frac{1}{3}(\mu_{\bar{\mathcal{M}}}(f\circ a)+\mu_{\bar{\mathcal{M}}}(f\circ b)+\mu_{\bar{\mathcal{M}}}(f\circ c))=$$
$$\frac{1}{3}\lim_{n\to\infty}\frac{1}{|\Gamma_n|}(S_n(f\circ a,0)+S_n(f\circ b,0)+S_n(f\circ c,0))=$$
$$\frac{1}{3}\lim_{n\to\infty}\frac{1}{|\Gamma_n|}\sum_{\gamma\in\Gamma_n}(f(a\gamma0)+f(b\gamma0)+f(c\gamma0))\pd$$
Lemma \ref{ReducedRep} implies that this is equal to
$$\frac{1}{3}\lim_{n\to\infty}\frac{1}{|\Gamma_n|}(S_{n+1}(f,0)+2S_{n-1}(f,0))=$$
$$\lim_{n\to\infty}\frac{1}{3}(2\frac{S_{n+1}(f,0)}{|\Gamma_{n+1}|}+\frac{S_{n-1}(f,0)}{|\Gamma_{n-1}|})=\frac{1}{3}(2\mu_{\bar{\mathcal{M}}}(f)+\mu_{\bar{\mathcal{M}}}(f))=\mu_{\bar{\mathcal{M}}}(f)\pd$$
By Riesz representation theorem measures are determined by their values over continuous functions, hence we are done.
\end{proof}
The fact that the stationary measure and the word metric limit coincide is somewhat surprising. The following theorem generalizes the conditions under this occurs.
\begin{definition}
Let $\mu$ be a probability measure on $G$ with $supp(\mu)=S\subset G$ such that $S$ generates $G$. We denote $G_n=\{g\in G: ||g|||_S =n\}$.
\begin{enumerate}
    \item $\mu$ is called evenly distributed if the mass that $\mu^{*n}$ assigns to an element depends only on its word metric. That is, for any $n,m\in\mathbbm{N}$ exists $1\ge\mu_{n,m}\ge 0$ such that any $g\in G$ with $||g||_S=m$ has $\mu^{*n}(g)=\mu_{n,m}$.
    \item $G$ action on a compact space $X$ is said to converge in word metric with respect to $S$ if for any $f\in C(X)$ and for any $x\in X$ the following limit exists:
    $$\lim_{n\to\infty}\frac{1}{|G_n|}\sum_{g\in G_n}f(gx)\pd$$
\end{enumerate}
\end{definition}
Notice that using Riesz representation theorem we know that if $G$ converges in word metric then it converges to a space average with respect to some measure $\nu$.
\begin{theorem}
\label{WordStationary}
Let $G$ be a group acting continuously on compact space $X$. Let $\mu$ be a probability measure on $G$ such that $S\defeq supp(\mu)$ generates G. Assume $\mu$ is evenly distributed and that $G$ action on $X$ converges in word metric with respect to $S$ to a space average with respect to $\nu$. Then $\nu$ is $\mu$ stationary.
\end{theorem}
To prove this theorem we make use of two lemmas. Denote by $\delta_x$ the Dirac measure at point $x$.
\begin{lemma}
\label{Gen1}
In the settings of Theorem \ref{WordStationary}, for any $x\in X$
$$\mu^{\ast n}\ast\delta_x \xrightarrow{\text{weak}-\ast} \nu\pd$$
\end{lemma}
\begin{proof}
Let $f\in C(X)$. Since $\mu$ is evenly distributed we can write
$$\mu^{*n}\ast  \delta_x(f)=\int_G f(\gamma x)d\mu^{*n}(\gamma)=\sum_{m=1}^n\sum_{\gamma \in \Gamma_m}\mu^{*n}(\gamma)f(\gamma x)=\sum_{m=1}^n\mu_{n,m}S_m(f,x)\pd$$
$G$ action on $X$ converges in word metric, hence we approximate $S_m(f,x)=|G_m|(\nu(f)+\epsilon(m))$ with $\epsilon(m)\xrightarrow[m\to\infty]{}0$. $\mu^{*n}$ is a probability measure hence $\sum_{m=1}^n\mu_{n,m}|G_m|=1$ and therefore
$$\sum_{m=1}^n\mu_{n,m}|G_m|(\nu(f)+\epsilon(m))=\nu(f)+\sum_{m=1}^n\mu_{n,m}|G_m|\epsilon(m)\pd$$
Let $M=\sup(\{\epsilon(m):m\in\mathbbm{N}\})$ and $k\in\mathbbm{N}$. The second term can be bounded by
$$\sum_{m=1}^n \mu_{n,m}|G_m|\epsilon(m)\le\mu^{*n}(\bigcup_{i\le k}G_i)M+\max(\{\epsilon(m):m> k\})\pd$$
Since for any $k$, $\lim_{n\to\infty}\mu^{*n}(\bigcup_{i\le k}G_i)=0$ we get
$$\lim_{n\to\infty}\mu^{*n}\ast  \delta_x(f)\le \nu(f)+\max(\{\epsilon(m): m> k\})\pd$$
$k$ is arbitrary hence we are done.
\end{proof}
\begin{lemma}
\label{Gen2}
Let $G$ be a group acting continuously on $X$. Let $\pi$ be a probability measure on $X$ and $\mu$ a probability measure on $G$. Assuming
$$\mu^{\ast n}\ast \pi\xrightarrow{\text{weak}-\ast}\nu$$
implies that $\nu$ is $\mu$ stationary.
\end{lemma}
\begin{proof}
let $\nu_n\defeq\mu^{\ast n}\ast \pi$. Since the space of measures is metric and since $\nu_n$ converges to $\nu$, the  Cesaro average $\frac{1}{n}\sum_{k=1}^n \nu_k$ converges to $\nu$ as well. The difference measure $\Delta_n\defeq\nu-\nu_n$ has $\lim_{n\to\infty}\Delta_n=0$. Then
$$|\mu\ast\nu-\nu|=|\mu * (\frac{1}{n}\sum_{k=1}^n \mu^{*k}*\pi + \Delta_n)- (\frac{1}{n}\sum_{k=1}^n \mu^{*k}*\pi + \Delta_n)|=$$
$$|\frac{\nu_{n+1}-\nu_1}{n}+\mu*\Delta_n-\Delta_n|\pd$$
Since $\nu_k$ are probability measures for all $k$ and since $\mu * \Delta_n$ tends to 0 we get
$$|\mu\ast\nu-\nu|=\lim_{n\to\infty}|\mu\ast\nu-\nu|=$$
$$\lim_{n\to\infty}|\frac{\nu_{n+1}-\nu_1}{n}+\mu*\Delta_n-\Delta_n|=0\pd$$
\end{proof}
The fact that the random walk converges both to a stationary measure and the word metric limit proves Theorem \ref{WordStationary}. One can check that the conditions of Theorem \ref{WordStationary} apply to the Farey group acting on the projective line thus Theorem \ref{MinkConv} follows as well.
\section{Lattice action on $\mathbbm{R}^2$}
\subsection{Settings and results}
The results presented in the previous chapter have shown that studying a random walk on a space can shed light on the word metric problem. We therefore proceed to study the asymptotical distribution problem for a random walk on the Euclidean plane. More precisely, we set a probability measure $\mu$ on $G=SL(2,\mathbbm{R})$ and ask if for a given point $x_0\in X = \mathbbm{R}^2\setminus \{0\}$there exists a normalization function $\Psi(n):\mathbbm{N}\to\mathbbm{R}$ such that the sequence $\Psi(n)\mu^{*n}*\delta_{x_0}$ converges in weak-$\ast$ topology, and if so to what measure. Notice that the convolution is defined with respect to the usual linear matrix action on the plane.\\\\ 
This problem has not been generally solved yet. We first suggest a variant of it that seems both natural and easier. Let $\lambda_1$ be the top Lyapunov exponent associated with $\mu$, defined by
$$\lambda_1 = \lim_{n\to\infty}\frac{1}{n}\int_{G}\log ||g||d\mu^{*n}(g).$$
We ask rather the sequence $\Psi(n)\mu^{*n}*\delta_{e^{-\lambda_1 n}x_0}$ converges in weak-$\ast$ topology to a space avarage $\bar{\nu}$. Equivalently, we ask if exists a normalization function $\Psi(n)$ such that for any $f\in C_c(X)$ the sequence
$$\lim_{n\to\infty}\Psi(n)\int_G f(e^{-\lambda_1 n}g x_0)d\mu^{*n}(g)$$
converges. Re-scaling using the top Lyapunov exponent is somewhat natural. Informally speaking, for a given $x_0$ almost every walk has $\frac{|g_1g_2...g_nx_0 - x_0|}{e^{n\lambda_1}} \to 1$. The suggested re-scale stops the points from drifting with exponential speed, thus makes it easier to study the structure of the resulting distribution. Notice that even with re-scaling almost every walk drifts to either $\infty$ or $0$, so the proportion of walks landing in any compact set out of all walks converges to $0$. Re-scaling by $e^{-n\lambda_1}$ only slows down the drift to a sub-exponential pace.\\\\
In the first section we shall show that under some assumptions on $G$ and assuming $\bar{\nu}$ can be decomposed to radial and angular measures, the measure $\bar{\nu}$ can be precisely described. Turns out that under these assumptions $\bar{\nu}$ is locally finite, infinite and stationary with respect to $\mu$. The main tool used in the above results is a recent central limit theorem by Benoist-Quint \cite{Benoist}. This theorem states that radial behavior of $\mu^{*n}*\delta_{x_0}$ can be approximated with a normal distribution with increasing mean and variance. Through the second part of this chapter we will explore what can be deduced if a stronger approximation assumption is being used.
\subsection{Theorem and proof}
We first define two properties needed to state and prove the main theorem.
\begin{definition}
Let $\mu$ be a measure on $G=SL(2,\mathbbm{R})$. Denote by $||g||$ the usual Euclidean norm on $G$ and by $G_\mu$ the closed semigroup spanned by the support of $\mu$.
\begin{enumerate}
    \item $\mu$ is said to have \textit{finite exponential moment} if exists $\alpha>$ such that
$$\int_G ||g||^\alpha d\mu(g)<\infty.$$
   \item $G_\mu$ is said to be \textit{strongly irreducible} if no proper finite union of vector subspaces in $\mathbbm{R}^2$ is $G_\mu$ invariant.
\end{enumerate}
\end{definition}

\begin{theorem}
\label{MainThm}
Let $\mu$ be a Borel probability measure on $SL(2,\mathbbm{R})$ with finite exponential moment such that $G_\mu$ is strongly irreducible and unbounded with respect to the Euclidean norm on $G$. Let $x_0\in X$ and assume $\Psi(n)\mu^{*n}*\delta_{e^{-\lambda_1 n}x_0}$ converges in weak-$\ast$ to $\bar{\nu}\neq0$, with $\Psi(n)$ being some normalization function. Further assume that $\bar{\nu}$ can be decomposed to a radial measure on $\mathbbm{R}^+$ and probability angular measures on $\mathbbm{P}^1$, that is $\bar{\nu}=\rho\otimes\nu $. Then $d\rho\propto\frac{1}{r}dr$ where $dr$ is the Lebesgue measure and $\nu$ is the unique $\mu$-stationary measure on $\mathbbm{P}^1$. In addition, $\lim_{n\to\infty}\frac{\Psi(n)}{\sqrt{n}}$ exists and is bigger then $0$.
\end{theorem}
Uniqueness of stationary measure is due to a theorem by Furstenberg that can be found in \cite{Benoist}. Theorem 16.10 in \cite{Benoist} is a key component in the proof. We bring an abbreviated version which is sufficient for our needs.
\begin{lemma}
\label{LLT}
Let $\mu$ be a Borel probability measure on $SL_2(\mathbbm{R})$ with finite exponential moment such that $G_\mu$ is unbounded and strongly irreducible. Let $a_1<a_2$ and $v\in\mathbbm{R}^2$ with $|v|=1$. Then exists $s\in\mathbbm{R}$ depending on $\mu,v$ such that
$$\lim_{n\to\infty}\frac{\mu^{*n}(\log (|gv|)-\lambda_1 n \in [a_1,a_2])}{N_{\sqrt{ns^2}}([a_1,a_2])}=1,$$
where $\lambda_1$ is the first Lyapunov exponent and $N_{\sqrt{ns^2}}$ is the normal distribution centered around 0 with standard deviation equal to $\sqrt{ns^2}$.
\end{lemma}
We are now ready to prove the first part of theorem \ref{MainThm} regarding the radial part of the measure.
\begin{lemma}
In the settings of theorem \ref{MainThm}, $d\rho \propto \frac{1}{r}dr$ where $dr$ is the Lebesgue measure.
\end{lemma}
\begin{proof}
Consider the limit
$$\bar{\nu}(f)=\lim_{n\to\infty}\Psi(n)\int_G f(e^{-\lambda_1 n}g x_0)d\mu^{*n}(g) .$$
We set $f = \mathbbm{1}_{D_{r,R}}$ with $R>r\in\mathbbm{R}$ and $D_{r,R}=\{x\in X:r\le|x|\le R\}$. The definition of $\bar{\nu}$ yields $$\bar{\nu}(f)=\lim_{n\to\infty}\Psi(n)\mu^{*n}(|e^{-\lambda_1 n}g x_0|\in [r,R])=$$
$$\lim_{n\to\infty}\Psi(n)\mu^{*n}(\log|g x_0|-\lambda_1 n\in[\log r,\log R])=$$
$$\lim_{n\to\infty}\Psi(n)N_{\sqrt{ns^2}}([\log r,\log R])
\frac{\mu^{*n}(\log|g x_0|-\lambda_1 n\in[\log r,\log R])}{N_{\sqrt{ns^2}}([\log r,\log R])}.$$
Since $\bar{\nu}\neq 0$ exists and as $0<\lim_{n\to\infty}\sqrt{n}N_{\sqrt{ns^2}}([\log r,\log R])<\infty$ we conclude that $0<\lim_{n\to\infty}\frac{\Psi(n)}{\sqrt{n}}<\infty$.  Using lemma \ref{LLT} we get
$$\bar{\nu}(f)= k\lim_{n\to\infty}\frac{1}{\sqrt{2\pi s^2}}\int_{\log r}^{\log R}\exp(-\frac{x^2}{2s^2n})=k\frac{\log R-\log r}{\sqrt{2\pi s^2}},$$
where $k\in\mathbbm{R}$ is some constant which depends on the choice of $\Psi(n)$. On the other hand
$$\bar{\nu}(f)=\int_r^R d\rho=\rho([r,R]),$$
and therefore $\rho([r,R])=\frac{k}{\sqrt{2\pi s^2}}(\log(R)-\log(r))$.
Intervals on $\mathbbm{R}^+$ are a generating algebra for the Borel $\sigma$-algebra so by unique extension we get that $d\rho \propto r^{-1}dr$.
\end{proof}
We shall now prove the claim regarding the angular part of the decomposition. This part is largely inspired by \cite{mac}.
\begin{lemma}
In the settings of theorem \ref{MainThm}, $\bar{\nu}$ is homogenous of degree 0. That is, for every measurable $E$ and for any $t>0$, it holds that $\bar{\nu}(tE)=\bar{\nu}(E)$
\end{lemma}
\begin{proof}
Take $E=A\times B$ with $A=[a,b]$ interval in $\mathbbm{R}^+$ and $B\subset \mathbbm{P}^1$ measurable. Then
$$\bar{\nu}(tE)=\nu(B)\int_{ta}^{tb} \frac{1}{x}dx=\nu(B)\int_a^b\frac{1}{ty}tdy=\nu(B)\rho(A)=\bar{\nu}(E)$$
Linear combinations of such box indicators are dense in indicators and so we are done.
\end{proof}
\begin{lemma}
$$\mu * \bar{\nu}=\bar{\nu}$$
\end{lemma}
\begin{proof}
$$\mu * \bar{\nu}(f)=\int_G\int_Xf(gv)d\bar{\nu}(v)d\mu(g)=$$
$$\int_G(\Psi(n)\int_G f(e^{-\lambda_1 n}ghv)d\mu^{*n}(h)+\Delta(n))d\mu(g)=$$
where $\Delta(n)=\int_{\mathbbm{R}^2}f(gv)d\bar{\nu}(v)-\int_G f(e^{-\lambda_1 n}ghv)d\mu^{*n}(h)$ has $\lim_{n\to\infty}\Delta(n)=0$ Then
$$\frac{\Psi(n)}{\Psi(n+1)}\Psi(n+1)\int_G f(e^{\lambda_1} e^{-\lambda_1 (n+1)}gv)d\mu^{*(n+1)}(g)+\Delta(n)=$$
define $h(x)=f(e^{\lambda_1} x)$:
$$\frac{\Psi(n)}{\Psi(n+1)}\Psi(n+1)\int_G h(e^{-\lambda_1 (n+1)}gv)d\mu^{*(n+1)}(g)+\Delta(n)=$$
$$\frac{\Psi(n)}{\Psi(n+1)}(\int_{\mathbbm{R}^2}h(v)d\bar{\nu}(v)+\Delta (n+1))+\Delta(n)=\frac{\Psi(n)}{\Psi(n+1)}(\bar{\nu}(f)+\Delta (n+1))+\Delta(n)$$
We can now take $n$ to $\infty$ to achieve
$$\mu * \bar{\nu}(f)=\lim_{n\to\infty}\mu * \bar{\nu}(f)=$$
$$\lim_{n\to\infty}\frac{\Psi(n)}{\Psi(n+1)}(\bar{\nu}(f)+\Delta (n+1))+\Delta(n)=\bar{\nu}(f).$$
\end{proof}
\begin{lemma}
If $\mu * \bar{\nu} = \bar{\nu}$ then $\mu * \nu = \nu$ 
\end{lemma}
\begin{proof}
Consider the radial integration operator $K:C_c(\mathbbm{R}^2)\to C(\mathbbm{P}^1)$ defined by integration over the real line with the following measure 
$$K(f)(\theta)=\int_\mathbbm{R}f(r,\theta)\frac{1}{r}dr.$$ 
It has $\nu(K(f))=\bar{\nu}(f)$. Let $\sigma(g,\theta)=\frac{|g v|}{|v|}$ be the size cocycle, with $g$ acting linearly on a vector with angle $\theta$. Using a simple change of variables $z=r\sigma(g,\theta)$ we see
$$K(f\circ g)(\theta)=\int_\mathbbm{R}f(r\sigma(g,\theta),g\theta)\frac{1}{r}dr=\int_\mathbbm{R}f(z,g\theta)\frac{1}{z}dz=K(f)(g \theta).$$
We then compute:
$$\nu(K(f))=\bar{\nu}(f)=\mu * \bar{\nu}(f)=\int_G \bar{\nu}(f\circ g)=$$
$$\int_G \nu(K(f\circ g)(\theta))=\int_G \nu(K(f)(g\theta))=\mu * \nu(K(f))$$
And since $K$ is surjective we are done.
\end{proof}
Theorem \ref{MainThm} relies on two assumptions, that $\Psi(n)\mu^{*n}\ast\delta_{e^{-\lambda_1n}x_0}$ converges, and that the limit measure $\bar{\nu}$ can be decomposed. These assumptions are not trivial at all and one should ponder whether exists a choice of $\mu,x_0 $ and $\Psi(n)$ for which they hold. The following numerical calculation suggest that these assumptions hold for the Farey lattice, which was studied in the first part of the research. We sample 1,000,000 walks with 120 steps each from the random walk in Theorem \ref{MinkStationary}. 
\begin{figure}[h]
\label{GraphComp}
\centering
\includegraphics[width=\textwidth]{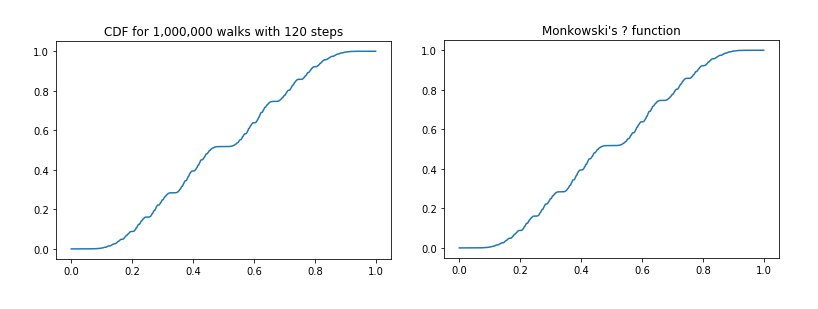}
\caption{Comparison between simulation CDF and Minkowski's ? function}
\end{figure}
We then act linearly with the sampled matrices on the vector $\icol{1\\0}$ and pull the resulting vectors to $\mathbbm{R}$ using  $\icol{x\\y}\xrightarrow{}\frac{x}{y}$. The plots in figure 4 compare the CDF for all vectors with size between 1 and 10,000 to Minkowski's question mark function. If the theorem holds, then the CDF should match the angular part of the decomposition. It is clear that the resulting CDF is identical to Minkowski's question mark function, which is the stationary measure in this case.
\subsection{Different regularization constant}
One can ask what would happen if the regularization constant has different value than $\lambda_1$. Unfortunately, the lack of suitable limit theorem prevents us from taking the same approach. We can still achieve some interesting results regarding the limit measure $\bar{\nu}$ under stricter convergence assumptions. While for theorem \ref{MainThm} there are convincing numerical results, this part is more of a shot in the dark. Hence it is written in a less rigorous fashion. This time we assume a stronger version of lemma $\ref{LLT}$.
\begin{definition}
\label{strong_to_normal}
We say a random walk generated by a measure $\mu$ over $SL(2,\mathbbm{R})$ \textit{strongly converges to normal} if for any $c
\in
\mathbbm{R}$,
$$\lim_{n\to\infty}\frac{\mu^{*n}(\log (|gv|)-\lambda_1 n \in [a_1,a_2]+cn)}{N_{\sqrt{ns^2}}([a_1,a_2]+cn)}=1.$$
\end{definition}
The special case $c=0$ was proved in \cite{Benoist}. Let $\alpha\in\mathbbm{R}$ and assume $
\mu$ strongly converges to normal and $\Psi(n)\mu^{*n}*\delta_{e^{-\alpha n}x_0}$ converges to some space average $\bar{\nu}$ which can be decomposed to radial and angular measures as before. We again use $D_{r,R}$ to see
$$\bar{\nu}(D_{r,R})=\lim_{n\to\infty}\Psi(n)\int_G\mathbbm{1}_{D_{r,R}}(e^{-\alpha n}g v)d\mu^{*n}=$$
$$\lim_{n\to\infty}\Psi(n)\mu^{*n}(\log|g v|-\lambda n\in[\log r, \log R]+(\alpha-\lambda)n).$$
Using definition \ref{strong_to_normal} we then get
$$\bar{\nu}(D_{r,R})=\lim_{n\to\infty}\Psi(n)N_{\sqrt{ns^2}}([\log r,\log R]+n(\alpha-\lambda_1))=$$
$$\lim_{n\to\infty}\Psi(n)P( Z\in\frac{1}{s\sqrt{n}}([\log r,\log R]-n(\lambda_1-\alpha)))=$$
$$\lim_{n\to\infty}\Psi(n)(P( Z\ge\frac{\log r-n(\lambda_1-\alpha)}{s\sqrt{n}})-P( Z\ge\frac{\log R-n(\lambda_1-\alpha)}{s\sqrt{n}})),$$
where $Z$ is the standard normal distribution with mean 0 and variance 1. For an approximation of $P(Z\ge\frac{\log r-n(\lambda_1-\alpha)}{s\sqrt{n}})$ we use a tail approximation
$$P(Z\ge x)=\frac{1}{\sqrt{2\pi}}\int_x^\infty e^{-t^2/2}dt\le \frac{1}{\sqrt{2\pi}}\int_x^\infty\frac{t}{x}e^{-t^2/2}dt=\frac{e^{-x^2/2}}{\sqrt{2\pi}x},$$
$$P(Z\ge x)>P(Z\in[x,x+\frac{1}{x}])\ge\frac{1}{x\sqrt{2\pi}}e^{-\frac{(x+
\frac{1}{x})^2}{2}}.$$
If $\alpha>\lambda_1$ the argument tends to infinity so we can approximate $P(Z\ge x)\approx \frac{1}{\sqrt{2\pi}x}e^{-x^2/2}$. If $\alpha < \lambda_1$ we can use the complement probability. Assume w.l.o.g that $\alpha>\lambda_1$ and let $b\defeq\frac{\lambda_1-\alpha}{s}$ and $a_r\defeq \frac{\log r}{s}$. Then
$$P(Z\ge\frac{\log r-n(\lambda_1-\alpha)}{s\sqrt{n}})\approx \frac{e^{-\frac{1}{2}(\frac{a_r^2}{n}-2a_rb+b^2n)}}{b\sqrt{n}-\frac{a_r}{\sqrt{n}}}\approx\frac{e^{a_rb-\frac{b^2n}{2}}}{b\sqrt{n}},$$
with the last approximation being true for large $n$. Since $b$ depends entirely on the choice of $\mu$ we see that a suitable choice for $\Psi(n)$ would be $b\sqrt{n}e^{\frac{b^2n}{2}}$ which leaves us with 
$$\bar{\nu}(D_{r,R})=\rho([r,R])=e^{a_Rb}-e^{a_rb}=R^{\frac{\lambda_1-\alpha}{s^2}}-r^{\frac{\lambda_1-\alpha}{s^2}}.$$
Using same reasoning as before we deduce $d\rho\propto r^{\frac{\lambda_1-\alpha}{s^2}-1}dr$ where $dr$ is the Lebesgue measure. Notice that when $\alpha>\lambda_1$ the resulting distribution $\bar{\nu}$ has $\bar{\nu}(B_r)=\infty$ and $\bar{\nu}(\bar{B}_r)<\infty$ for $B_r=\{x\in\mathbbm{R}^2\setminus\{0\}:x<r\}$ for any $r>0$. In the case $\alpha<\lambda_1$ we get $\bar{\nu}(\bar{B}_r)=\infty$ and $\bar{\nu}(B_r)<\infty$ for any $r>0$. $\alpha=\lambda_1$ is the unique case where both $\bar{\nu}(B_r)=\infty$ and $\bar{\nu}(\bar{B}_r)=\infty$.

\end{document}